\documentclass[11pt,a4paper]{amsart}
\usepackage{mathrsfs}
\usepackage{amsfonts}
\usepackage[notref,notcite]{showkeys}
\usepackage{enumerate}

\usepackage{amsmath,amssymb,xspace,amsthm}

\newtheorem{theorem}{Theorem}
\newtheorem{lemma}[theorem]{Lemma}
\newtheorem{corollary}[theorem]{Corollary}
\numberwithin{equation}{section}
\newcommand{\thmref}[1]{Theorem~\ref{#1}}

\newcommand{\Corref}[1]{Corollary~\ref{#1}}

\newcommand{\lemref}[1]{Lemma~\ref{#1}}

\def\Der{\operatorname{Der}}
\def\Aut{\operatorname{Aut}}

\newcommand{\C}{\ensuremath{\mathbb C}\xspace}

\renewcommand{\a}{\ensuremath{\alpha}}

\newcommand{\N}{\mathbb{N}}
\newcommand{\Z}{\ensuremath{\mathbb{Z}}\xspace}

\newcommand{\supp}{\ensuremath{\operatorname{Supp}}\xspace}

\renewcommand{\phi}{\varphi}

\renewcommand{\leq}{\leqslant}
\renewcommand{\geq}{\geqslant}

\newcommand{\Hom}{\operatorname{Hom}}
\newcommand{\Deg}{\operatorname{Deg}}
\newcommand{\id}{\operatorname{id}}

\newcommand{\p}{\partial}

\title{Simple superelliptic Lie algebras}

\author[1]{Ben Cox}\author{Xiangqian Guo}\author{Rencai Lu}\author{Kaiming Zhao}
\address{Department of Mathematics \\
University of Charleston, South Carolina \\
66 George Street  \\
Charleston SC 29424, USA}\email{coxbl@cofc.edu}

\address{Department of Mathematics \\
 Zhengzhou
University  \\
Zhengzhou 450001, Henan
\\P. R. China.}\email{guoxq@zzu.edu.cn}

\address{Department of Mathematics \\
Suzhou University\\
  Suzhou 215006, Jiangsu \\
   P. R. China}
\email{rencail@amss.ac.cn}

\address{Department of Mathematics\\
 Wilfrid
Laurier University \\ Waterloo, ON \\
Canada N2L 3C5\\
 and College of
Mathematics and Information Science \\
Hebei Normal (Teachers)
University \\  Shijiazhuang, Hebei, 050016  \\ P. R. China. }
\email{kzhao@wlu.ca}

\keywords{Krichever Novikov Algebras, Automorphism Groups, Pell's Equation, associated Legendre polynomials, universal central extensions, superelliptic Lie algebras, superelliptic curves, DJKM algebras, F\'aa di Bruno's formula, Bell polynomials}

\begin{document}

\begin{abstract}
Let  $m\in\N$, $P(t)\in\C[t]$.   Then we have the Riemann surfaces (commutative algebras)
$R_m(P)=\C[t^{\pm1},u \,|\, u^m=P(t)]$ and $S_m(P)=\C[t , u\,|\, u^m=P(t)].$ The Lie algebras  $\mathcal{R}_m(P)=\Der(R_m(P))$ and $\mathcal{S}_m(P)=\Der(S_m(P))$ are called the $m$-th
superelliptic Lie algebras associated to $P(t)$.
 In this paper we  determine the necessary and sufficient conditions for such Lie
algebras to be simple, and determine their universal central extensions and their derivation algebras. We also study the  isomorphism and automorphism problem for these Lie algebras (the Riemann surfaces) by using polynomial Pell equations.
\end{abstract}
\maketitle

\section{Introduction}

Throughout this paper we will take the set of natural numbers to be $\mathbb N=\{1,2,\dots\}$, the set of nonnegative integers will be denoted by $\mathbb
Z_+=\{0,1,2,3,\dots\}$, and we will assume all vector spaces and algebras are defined over
complex numbers $\C$.

One can consider the Laurent polynomial ring $\C[t, t^{-1}]$ as the ring
of rational functions on the Riemann sphere $\C \cup \{\infty\}$
with poles allowed only in $\{\infty, 0\}$. This geometric point of
view suggests a natural generalization of the loop algebra
construction. Instead of the sphere with two punctures, one can
consider any complex algebraic curve $X$ of genus $g$ with a fixed
subset $P$ of $n$ distinct points. Following this idea one arrives at M. Schlichenmaier's definition of multipoint algebras of Krichever-Novikov affine type if we replace $\C[t, t^{-1}]$ with the ring $R$ of
meromorphic functions on $X$ with poles allowed only in $P$ in the
construction of affine Kac-Moody algebras (see \cite{MR2058804},  \cite{MR902293}, \cite{MR925072}, and \cite{MR998426}). The $n$-point affine Lie algebras which are a type of Krichever-Novikov algebra of genus zero also appeared in the
work of Kazhdan and Lusztig (\cite[Sections 4 \& 7]{MR1104840}, \cite[Chapter 12]{MR1849359}).  Krichever-Novikov algebras are used to constuct analogues of important mathematical objects used in string theory but in the setting of a Riemann surface of arbitrary genus. Moreover Wess-Zumino-Witten-Novikov theory and analogues of the Knizhnik-Zamolodchikov equations are developed for analogues of the affine and Virasoro algebras (see the survey article \cite{MR2152962}, and for example  \cite{MR1706819}, \cite{MR2106647}, \cite{MR2072650}, \cite{MR2058804}, \cite{MR1989644}, and \cite{MR1666274}).

In an earlier paper \cite{MR3211093} we introduced and studied the $n$-point Virasoro algebras, $\tilde{\mathcal{V}}_a$,
 which are natural generalizations
 of the classical Virasoro algebra and have as quotients multipoint genus zero Krichever-Novikov type algebras. These algebras are the universal central extension  of the derivation Lie algebras of  the Riemann sphere with $n$-points removed. We determined necessary and sufficient conditions for  two such Lie algebras
 to be isomorphic.  Moreover we determined their automorphisms, their derivation algebras, their universal
 central extensions,  and some other
 properties.   The  list of automorphism groups that occur is $C_n$, $D_n$, $A_4$, $S_4$ and $A_5$. We also constructed a large class of modules which we call modules of densities, and
 determine necessary and sufficient conditions for them to be irreducible.

In the present paper we turn to the study of Riemann surfaces of positive genus.  The particular class we look at are  superelliptic curves.  Let $P(t)\in\C[t]$, $m\in\N$. Then we have the Riemann surfaces (commutative associative algebras) $R_m(P)=\C[t^{\pm1},u \,|\, u^m=P(t)]$ and $S_m(P)=\C[t , u\,|\, u^m=P(t)].$ The Lie algebras  $\mathcal{R}_m(P)=\Der(R_m(P))$ and $\mathcal{S}_m(P)=\Der(S_m(P))$ are called  {\it superelliptic, (respectively hyperelliptic) Lie
algebras} due to the fact that $u^m=P(t)$ is a superelliptic (respectively hyperelliptic) curve if $m>2$ (resp. $m=2$).  These algebras are of Krichever-Novikov type.

In the second section one of our main results is given in
\thmref{der.a} where we derive necessary and sufficient condition in
terms of root multiplicities of $P(t)$, for $\mathcal{R}_m(P)$ and
$\mathcal{S}_m(P)$ to be  simple infinite dimensional Lie algebras.
The proof relies on Jordan's  results from \cite{MR829385,MR1764580}. We also deduce
in \thmref{thm7} that all derivations of the simple Lie algebras
$\mathcal{R}_m(P)$ and $\mathcal{S}_m(P)$ are inner.

 In the third section, we   will mainly  use the
 results from the paper \cite{MR2035385} by Skyabin to  determine the universal
 central extension of the Lie algebras
 $\mathcal{R}_m(P)$ and $\mathcal{S}_m(P)$, and in particular we obtain a basis of  the 2-cocycles of the Lie algebras $\mathcal{R}_m(P)$ and
$\mathcal{S}_m(P)$, see \thmref{der.b} and 11.
We also explicitly give an example showing how to compute the value of any 2-cocyle on a basis of $\mathcal{R}_m(P)$.  
 Having such an explicit description of the two cocycles will allow, one using conformal
 field theoretic tools, to study free field type representations of these algebras.

To study isomorphisms and automorphisms between the Lie algebras
$\mathcal{R}_m({P})$ and $\mathcal{S}_m({P})$, from \cite{MR966871}
we know that it is  equivalent  to considering isomorphisms and
automorphisms of the Riemann surfaces  ${R}_m({P})$ and
${S}_m({P})$. This is  a very hard problem.In particular it is known that there are 12 sporadic simple groups that can appear. For some results on higher genus Riemann surfaces, see \cite{MR2203507, MR1796706} and the reference therein.
  So in the last two sections we make  attempts  to study
isomorphisms (or automorphisms) between some of the hyperelliptic
Lie algebras $\mathcal{R}_2({P})$ and $\mathcal{S}_2({P})$ whose
corresponding Riemann surfaces ${R}_2({P})$ and ${S}_2({P})$ have
higher genus in general.

  In the fourth section, we describe the group of units of $R_2(P)$ and ${S}_2({P})$ which will be used in the fifth section,  in particular are able to explicitly describe this group in cases of $P(t)=t(t-a_1)\cdots (t-a_{2n})$ and $P(t)=t^4-2bt+1$, $b\neq \pm 1$ which is the most interesting case studied by Date, Jimbo, Kashiwara and Miwa \cite{DJKM} where they investigated integrable systems arising from Landau-Lifshitz differential equation.  When determining the unit group we find that it requires one find solutions of the polynomial Pell equation $f^2-g^2P=1$
for a given $P\in\C[t]$ which is a very famous and very hard problem.
See \cite{MR2183270}.
 The last section is devoted to determining the
conditions for $\mathcal{R}_m({P_1})\simeq
\mathcal{R}_m({P_2})$, in the case of $m=2$ and $P_1(t)$ and $P_2(t)$ are separable polynomials of odd degree with one root at $t=0$. Necessary and sufficient conditions are given in \thmref{isothm}.   Then \Corref{automorphismthm} describes the possible automorphism groups of $\mathcal{R}_2(P)$ as being either the trivial group, $D_k\times \mathbb Z_2$, or $\mathbb Z_k\times \mathbb Z_2$.
One may want to compare this to the description of the conformal automorphism group of a hyperellptic curve given in \cite{MR1223022} and \cite{MR2035219}.

Note if $m=1$ or $P(t)\in\C\setminus\{0\}$, we know that $R_m(P)\simeq
\C[t_1^{\pm1}]\oplus \dots \oplus  \C[t_k^{\pm1}]$ and $S_m(P)\simeq
\C[t_1]\oplus \dots \oplus  \C[t_k]$ for some $k\in\N$.
The corresponding Lie algebras are well-known. So in this paper we
always assume that $m\ge 2$ and $P(t)\in\C[t]\setminus \C$.

\section{Simplicity  of the Lie algebras $\mathcal{R}_m(P)$ and $\mathcal{S}_m(P)$}

We first recall a result on the simplicity of Lie algebras of
derivations from \cite{MR1764580}. Let $R$ be a commutative
associative algebra over $\C$ and $L\subseteq \Der(R)$ a nonzero Lie
subalgebra which is also a left $R$-submodule of $\Der(R)$. An ideal
$I\subseteq R$ is called $L$-stable if $I$ is stable under the
action of any derivation in $L$ and $R$ is called $L$-simple if $R$
does not have any nonzero $L$-stable ideal other than $R$.

\begin{lemma} \label{D-simple} The Lie algebra
$L$ is simple if and only if $R$ is $L$-simple.
\end{lemma}

A commutative ring $R$ is called regular if its localization $R_p$
at any prime ideal $p$ is a regular local ring (see \cite{M} page
140). By the example given on Page 37--38 of \cite{MR829385}, we
have

\begin{lemma} \label{regular}
Let $R=\C[x_1,\cdots,x_n]/\langle f\rangle$ for some irreducible polynomial $f$
and let $f_i$ be the image of $\frac{\p f}{\p x_i}$ in $R$ for any
$i=1,\cdots,n$.
\begin{itemize}
\item[(1)] $R$ is regular if and only if $R=\langle f_1,\cdots,f_n\rangle,$ the ideal of $\C[x_1,\cdots,x_n]$
generated by $f_1,\cdots, f_n$.
\item[(2)] $\Der(R)$ is simple if and only if $R$ is regular.
\end{itemize}
\end{lemma}

Now we begin to define our algebras. Let $m\in\N$, $P(t)\in\C[t]$,
and $\langle u^m-P(t) \rangle$ be the ideal of the polynomial
algebra $\C[t^{\pm1}, u]$ (or  $\C[t, u]$) generated by $u^m-P(t)$.
Then we have the Riemann surfaces
$$R_m(P)=\C[t^{\pm1}, u]/\langle u^m-P(t)
\rangle \,\,\,\,{\rm and } \,\,\,\, S_m(P)=\C[t , u]/\langle u^m-P(t)
\rangle.$$ We call the derivation
algebras $$\mathcal{R}_m(P)=\Der(R_m(P))  \,\,\,\,{\rm and } \,\,\,\, \mathcal{S}_m(P)=\Der(S_m(P))$$ {\it the $m$-th
superelliptic Lie algebras associated to $P(t)$}.

It is easy to see that
$$\aligned R_m(P)=&\C[t^{\pm1}]\oplus  \C[t^{\pm1}]u\oplus \cdots\oplus  \C[t^{\pm1}]u^{m-1}, \\
S_m(P)=&\C[t ]\oplus  \C[t]u\oplus \cdots\oplus  \C[t] u^{m-1}.\endaligned$$
For convenience, we denote  $\Delta=P'\frac \p{\p u}+mu^{m-1}\frac \p{\p t}$
where $P'=\frac {\p P}{\p t}$.
First we determine elements in the Lie algebras $\mathcal{R}_m(P)$ and $\mathcal{S}_m(P)$.

\begin{lemma} \label{der} Let $a(t)=\gcd(P,P')$. Then
$$\mathcal{S}_m(P)=
\C[t]\Delta+\sum_{i=1}^{m-1}\C[t]u^i\frac \Delta a.$$
$$\mathcal{R}_m(P)=
\C[t^{\pm1}]\Delta+\sum_{i=1}^{m-1}\C[t^{\pm1}]u^i\frac \Delta a.$$
If further $a=1$ then
$\mathcal{S}_m(P)=S_m(P)\Delta$; if  $a$ is invertible in $\C[t^{\pm1}]$ then $
\mathcal{R}_m(P)= R_m(P)\Delta.$
\end{lemma}

\begin{proof}
Suppose $D\in \mathcal{S}_m(P)$. Then $D$ is uniquely determined by
$D(t)$ and $D(u)$ since $S_m(P)$ is generated by $t, u$. We may
assume that $D= f\frac \p{\p t} +h\frac \p{\p u}  $ where
$h=\sum_{i=0}^{m-1}h_iu^i$, $f=\sum_{i=0}^{m-1}f_iu^i$ with
$h_i,f_i$ being polynomials in $t$. Next we will find the
restrictions on $f, h$. It is clear that $D\in \mathcal{S}_m(P)$  if and
only if $D(u^m-{P})=0$,  if and only if $-fP'(t)+mhu^{m-1}=0$, if
and only if
$$\sum_{i=0}^{m-1}P'f_iu^i=m\sum_{i=0}^{m-1}h_iu^{m+i-1}=m\sum_{i=0}^{m-2}h_{i+1}Pu^{i}+mh_0u^{m-1},
$$
 if and only if
 $$mh_0=P'f_{m-1},\quad mh_{i+1}P=P'f_i, \quad 0\le i\le m-2,$$
  if and only if there exist $g_i\in \C[t]$ such that
  $$mh_0=P'f_{m-1}, \quad mh_{i+1}=P'g_i/a, \quad f_i=Pg_i/a, \quad  0\le i\le m-2,$$
if and only if $$mD=f_{m-1}\Delta+   \sum_{i=0}^{m-2} g_iu^{i+1}\frac{\Delta} a \in
\C[t]\Delta+\sum_{i=1}^{m-1}\C[t]u^i\frac \Delta a.$$

Using similar arguments, we can prove that
$$\mathcal{R}_m(P)=
\C[t^{\pm1}]\Delta+\sum_{i=1}^{m-1}\C[t^{\pm1}]u^i\frac \Delta a.$$
 The rest of the lemma is clear.\end{proof}

As we mentioned at the end of Sect.1, the associative algebras
$R_m(P)$, $S_m(P)$ are clear when $m=1$ or $P\in \C$, and so are the
Lie algebras $\mathcal{R}_m(P)$, and $\mathcal{S}_m(P)$. Thus, from
now on we always assume that $m\ge 2$ and $\deg(P(t))\ge1$.

\begin{theorem} \label{der.a}  Suppose that $m\ge 2$ and $\deg(P)\ge1$.
\begin{itemize}\item[(a).] The Lie algebra $\mathcal{S}_m(P)$ is  simple if and only if
$P(t)$  does not have multiple roots.
\item[(b).] The Lie algebra $\mathcal{R}_m(P)$ is  simple if and only if
$P(t)$ has no multiple nonzero roots and $P(t)\ne ct^r$ for any $c\in \C^*$ and $r\in \Z_+$ with $(r,m)\ne 1$.\end{itemize}
\end{theorem}

\begin{proof} (a). $``\Longrightarrow".$ First suppose that  $P(t)$ has no multiple roots. Then $P(t)$ and
$P'(t)$ are relatively prime, and $u^m-P(t)$ is irreducible in
$\C[u,t]$. There exist $a(t), b(t)\in \C[t]$ such that $aP+bP'=1$.

 We want to show that $\mathcal{S}_m(P)$ is simple. To this end we need to show that $S_m(P)$ is
regular by Lemma \ref{regular}, which is equivalent to the condition
that the ideal of $S_m(P)$ generated by
$f_1=\frac{\partial}{\partial t}f=-P'(t)$ and $
f_2=\frac{\partial}{\partial u}f=mu^{m-1}$ is $S_m(P)$, where
$f=u^m-P$.

Since $-bf_1+\frac1m uaf_2=1$ then the ideal generated by $f_1, f_2$
is $S$. So $S$ is regular and hence $\Der(S)$ is simple.

$``\Longleftarrow".$  Now suppose $P(t)$  has a multiple root
$\lambda$. Then  $(t-\lambda) | a(t)$ where $a(t)=\gcd(P,P')$. We
want to show that $\mathcal{S}_m(P)$ is not simple. Note that
$(t-\lambda)a(t)|P(t)$. It is not hard to verify that
$I=(t-\lambda)\C[t]+\sum_{i=1}^{m-1}\C[t]u^i$ is a proper
$\mathcal{S}_m(P)$-stable ideal of $S_m(P)$. Thus $\mathcal{S}_m(P)$
is not simple  in this case by Lemma \ref{D-simple}.

(b). $``\Longleftarrow".$ Now suppose $P(t)$  has a nonzero multiple
root $\lambda$. Then  $(t-\lambda) | a(t)$ where $a(t)=\gcd(P,P')$.
We want to show that $\mathcal{R}_m(P)$ is not simple. Note that
$(t-\lambda)a(t)|P(t)$.  It is not hard to verify that
$I=(t-\lambda)\C[t^{\pm1}]+\sum_{i=1}^{m-1}\C[t^{\pm1}]u^i$ is a
proper $\mathcal{R}_m(P)$-stable ideal of $R_m(P)$. Thus
$\mathcal{R}_m(P)$ is not simple  in this case by Lemma
\ref{D-simple}.

If $P=(ct)^r$ for some $c\in \C^*$ and $r\in \Z_+$ with
$d=\gcd(r,m)>1$, without loss of generality we may assume that
$c=1$. Then $\mathcal{R}_m(P)={R}_m(P)\Delta$ where
$\Delta=rt^{r-1}\frac{\partial}{\partial
u}+mu^{m-1}\frac{\partial}{\partial t}$. Let $x=1-t^{-r/d}u^{m/d}\in
R_m(P)$. Then $\Delta(x)=0$. Note that $x |(u^m-P)$ in $\C[t^{\pm
1}, u]$, which implies $xR_m(P)$ is a $\mathcal{R}_m(P)$-stable
nonzero proper ideal of $R_m(P)$. Thus $\mathcal{R}_m(P)$ is not
simple either in this case.

$``\Longrightarrow".$ 
%
Now we consider the
remaining case that $P\ne ct^r$ for any $c\in \C^*$ and $r\in \Z_+$
with $d=(r,m)>1$, and $P$ does not have multiple nonzero roots.
So $\gcd(P, P')$ is invertible in $R_m(P)$. From \lemref{der} we know that $\mathcal{R}_m(P)= R_m(P)\Delta$, and an ideal of   $R_m(P)$ is $\Delta$-stable iff it is   ${R}_m(P)$-stable.

Suppose $I$ is a nonzero ideal of $R_m(P)$ that is $\Delta$-stable.
 For
$$x=\sum_{i=0}^{m-1}x_iu^i\in R_m(P)\text{ with }
x_i\in\C[t^{\pm1}],$$ we define the $i$-th support of $x$ as
$x_i=\supp_i(x)$ and define $\supp(x)=\{i|x_i\ne0\}$.
 Let
$$I_i=\{f\in \C[t^{\pm1}]\,|\,f=\supp_i(x)\text{ for some }x\in
I\}.$$ Clearly, each $I_i$ is a nonzero ideal of $\C[t^{\pm 1}]$
and
$$I_0\subset I_1\subset \cdots\subset I_{m-1}.$$

{\bf Claim 1.} {\it $I_i=R_m(P)$ for each $i$.}

Since $\C[t^{\pm 1}]$ is a principal ideal domain, there exists monic $f_i\in \C[t]$ with nonzero constant term such that
$I_i= f_i\C[t^{\pm 1}]$. Then we have $$f_0|Pf_{m-1},\quad
f_i|f_{i-1} \text{ for each } 1\le i\le m-1.$$

From $\Delta(f_iu^i)=iP'f_iu^{i-1}+mu^{m+i-1}f_i'$ we know that
$$f_{m-1}|f_0',\quad f_{i-1}|iP'f_i+mPf_i' \text{ for } 1\le i\le m-1.$$
Combining with $f_i|f_{i-1}$ we see that $f_i|Pf_i'$ for $0\le i\le
m-1$, yielding that any root of $f_i(t)$ is a root of $P(t)$ for
each $0\le i\le m-1$.

Next we show that $f_i=f_{i-1}$ for all $0\le i\le m-1$. Let $b$ be
a root of $f_i$ with multiplicity $k_i$ which is also a simple root
of $P$. From $f_i|f_{i-1}$ we know that $k_i\le k_{i-1}$. Write
$P=(t-b)P_1$ and $f_i=(t-b)^{k_i}g_i$. Then $t-b\nmid P_1g_i$. From
$f_{i-1}|iP'f_i+mPf_i'$ and
$$iP'f_i+mPf_i'=(i+mk_i)P_1g_i(t-b)^{k_i}+(t-b)^{k_i+1}(iP'_1g_i+mP_1g'_i)$$
we know that $k_i\ge k_{i-1}$. Thus $f_0=f_1=\cdots =f_{m-1}$.
Combining with $f_{m-1}|f_0'$ we see that $f_0=f_1=\cdots=
f_{m-1}=1$. Claim 1 follows.

\

{\bf Claim 2.} {\it $J_i:=I\cap R_m(P)u^i\ne0$ for each $i$.}

Let $k$ be minimal such that there exists nonzero $x\in I$ with
$\supp(x)\subset K=\{0,1,\cdots,k\}$. Let $J=\{x\in I\,|\,
\supp(x)\subset K\}$. There is  $x=g_0+g_1u+\cdots+g_ku^k\in J$ with
$g_i\in \C[t^{\pm1}]$ such that $J=\C[t^{\pm1}]x$.

If $k>0$, then $g_0g_k\ne0$. We have
$u\Delta(x)=mg_0'P+\sum_{i=1}^k(iP'g_i+mPg'_i)u^{i}\in J$,  yielding
$$\frac{mg_0'P}{g_0}=\frac{iP'g_i+mPg'_i}{g_i}, \,\, {\text {i.e.}}$$
$$\frac d{dt}\Big(\frac {g_i^mP^i}{g_0^m}\Big)=0.$$
So $\frac {g_k^mP^k}{g_0^m}$ is a nonzero constant. This is clearly impossible if $P$ has a simple nonzero root. If $P=ct^s$ with $\gcd(s, m)=1$,   this is also
impossible. So $k=0$ and Claim 2 follows.

It is easy to se that  $I':=\sum_{i=0}^{m-1}J_i\subset I$ is a
$\Delta$-stable ideal of $R_m(P)$. Applying Claim 1 to $I'$ we
deduce that $I'=R_m(P)$, hence,  $I=R_m(P)$. Consequently, $R_m(P)$
is $\Delta$-simple and $\mathcal{R}_m(P)$ is simple by Lemma
\ref{D-simple}. This completes the proof of the theorem.
\end{proof}

When $P=ct^r$ for $c\in\C^*, r\in\N$ with $\gcd(m,r)=1$, one can
easily see that $\mathcal{R}_m(P)$ is isomorphic to the classical
centerless Virasoro algebra.

From now on
we always assume that $m\ge2$ and that $P(t)$ does not have multiple nonzero
roots and has at least one nonzero root. So $\mathcal{R}_m(P)$ is a simple Lie algebra and not isomorphic to
the classical
centerless Virasoro algebra.

 We can realize  ${R}_m(P)$  as $\C[t^{\pm1},\sqrt[m] P]$ by identifying $u$ with $\sqrt[m] P$. So $R_m(P)=\C[t^{\pm1},\sqrt[m] P]$ and $u=\sqrt[m]
P$. Then  $\mathcal{R}_m(P)=R_m(P)\p$ where $\p=\sqrt[m]{P(t)^{m-1}}\frac d{dt}$.

\vskip 5pt {\bf Example 1.} We know that  any element in $\mathcal{R}_2(P)$ is of the form
$(f\sqrt{P}+gP)\frac d{dt}$ where $f,g\in\C[t^{\pm1}]$. For
convenience, we write $\partial_1=\sqrt{P} \frac d{dt}$ and
$\partial_2= {P} \frac d{dt}$. Then $\mathcal{R}_2(P)$ has a
natural basis $\{t^i\partial_1, t^i\partial_2\,|\,i\in\Z\}$ with
brackets:
\begin{equation}\label{bracket1}[t^i\partial_1, t^j\partial_1]=(j-i)t^{i+j-1}\partial_2,
\end{equation}
\begin{equation}\label{bracket2}[t^i\partial_2, t^j\partial_2]=(j-i)t^{i+j-1}P\partial_2,
\end{equation}
\begin{equation}\label{bracket3}[t^i\partial_1, t^j\partial_2]=(j-i)t^{i+j-1}P\partial_1+\frac12
t^{i+j}P'\partial_1.\qedhere
\end{equation}
If we let $L_0=R_2(P)\partial_2$ and $L_1=R_2(P)\partial_1$, then $\mathcal{R}_2(P)$ becomes a $\Z_2$-graded Lie algebra with even part $L_0$, and odd part $L_1$. One can easily see that $L_0$ is not a simple Lie algebra and $L_1$ is not a simple $L_0$-module.

\

 In the next
lemma we will see that  $\mathcal{R}_2({P})$ is actually a $3$- or
$4$-point Virasoro algebra if $P(t)$ is of degree  $1$ or $2$
respectively. For $n$-point Virasoro algebras we have systematical
studies in [BGLZ].

\begin{lemma}
\begin{itemize}\item[(a).] If $P(t)=t$, then $R_2(P)\simeq \C[t^{\pm1}].$
\item[(b).] If $P(t)=t-a^2$ where $a\ne0$, then
$R_2(P)\simeq \C[t^{\pm1}, \frac1{t+1}].$
\item[(c).]  If $P(t)=t^2-2bt+1$ where $b\ne\pm1$, then $$R_2(P)\simeq \C[t^{\pm1},
\frac1{t+b+1},\frac1{t+b-1}].$$
\item[(d).]  If $P(t)=t^2-2bt$ where $b\ne0$, then $R_2(P)\simeq \C[t^{\pm1},
\frac1{t +1}].$
\end{itemize}
\end{lemma}

\begin{proof} (a) is clear.

(b). $R_2(P)\simeq\C[t^{\pm1},u]/\langle u^2-t-a^2
\rangle\simeq\C[t^{-1},u]/\langle t-(u^2-a^2) \rangle$
$$\simeq\C[u,\frac1{u^2-a^2}] \simeq\C[t,\frac1{t^2-a^2}]\simeq\C[t,\frac1{(t+a)(t-a)}]$$
$$ \simeq\C[t,\frac1{t(t+2a)}] \simeq\C[t,\frac1{t},\frac1{t+2a}]$$
$$ \simeq\C[t^{\pm},\frac1{t+2a}]\simeq \C[t^{\pm1},
\frac1{t+1}].$$

(c). $R_2(P) \simeq\C[t^{\pm1},u]/\langle u^2-(t^2-2bt+1) \rangle$
$$\simeq\C[t^{\pm1},u]/\langle (u^2-(t-b)^2-(1-b^2) \rangle$$ $$\simeq\C[t^{\pm1},u]/\langle (u+t-b)(u-t+b)-(1-b^2) \rangle
$$
$$ \simeq \C[t^{-1}, v,v^{-1}]/\langle t^{-1}-\frac2{v-(1-b^2)v^{-1}+2b} \rangle$$
$${\text{ (We have let }} v=u+t-b, v^{-1}=(u-t+b)/(1-b^2)\, )$$
$$ \simeq \C[v,v^{-1},\frac1{v-(1-b^2)v^{-1}+2b}]$$
$$ \simeq \C[v,v^{-1},\frac1{v^2+2bv+b^2-1}]$$
$$\simeq \C[t^{\pm1},
\frac1{t+b+1},\frac1{t+b-1}].$$

(d) Similar to (c), we deduce that $$R_2(P)
\simeq\C[t^{\pm1},u]/\langle u^2-(t^2-2bt) \rangle$$
$$\simeq\C[t^{\pm1},u]/\langle (u-t+b)(u+t-b)+b^2 \rangle$$
$$ \simeq \C[v,v^{-1},\frac1{(v+b)^2}]\simeq \C[t^{\pm1},
\frac1{t+1}],$$
where $v=u+t-b$ and $v^{-1}=-(u-t+b)/b^2$.
\end{proof}

We point out that Part (c) in the last corollary was proved in [1]
with a different approach. With similar arguments as in the above
lemma we have the following

\begin{lemma} Suppose that $P(t)\in\C[t]$ has no multiple roots.
\begin{itemize}
\item[(a).] If $\deg(P(t))=1$, then $S_m(P)\simeq \C[t]$ for any $m\in\N$.
\item[(b).]  If $\deg(P(t))=2$, then $S_2(P)\simeq \C[t, t^{-1}].$
\end{itemize}
\end{lemma}

If $\deg(P)\ge3$ (keep in mind the conditions assumed for $P$), we
do not know whether the Lie algebras  $\mathcal{S}_m(P) $ and
$\mathcal{R}_m(P) $ are $n$-point Virasoro algebras.

Now we can
obtain all derivations of $\mathcal{S}_m(P) $ and $\mathcal{R}_m(P)
$.

\begin{theorem}\label{thm7} Suppose $\mathcal{S}_m(P) $ and  $\mathcal{R}_m(P) $ are simple Lie algebras. Then
all derivations of $\mathcal{S}_m(P) $ and  $\mathcal{R}_m(P) $ are inner derivations.
\end{theorem}

\begin{proof}
This is a direct consequence of [25, Theorem 3.2].
\end{proof}

\section{Universal central extensions of $\mathcal{R}_m(P) $ and $\mathcal{S}_m(P)$}

\subsection{General theory} Recall that  the Lie algebras $\mathcal{R}_m(P)$  and
$\mathcal{S}_m(P)$  have been assumed to be  simple. In this section
we will determine the universal central extension of
$\mathcal{R}_m(P) $ and $\mathcal{S}_m(P)$. We will mainly  use the
results from the paper \cite{MR2035385} by Skyabin.

Let us recall some notions from \cite{MR2035385}. Let $R=R_m(P)$ and
$W=\mathcal{R}_m(P)$, which is an $R$-module. We have the de Rham
complex relative to $W$
$$\Omega:\ \Omega^0\rightarrow \Omega^1\rightarrow\Omega^1\wedge\Omega^1\cdots,$$
where $\Omega^0=R$ and $\Omega^1=\Hom_R(W, R)$. For each $f\in R$,
define $df\in \Omega^1$ by the rule $df(D)=D(f)$ for any $D\in W$.
Then $\Omega^1$ can be viewed as a $W$-module via
$\rho_{\Omega^1}(D)(\theta)=d(\theta(D))$ for $\theta\in \Omega^1$
and $D\in W$ (See page 71 and page 105 of \cite{MR2035385}).
We will first verify   Conditions (1.1), (1.2) and (1.3) assumed in \cite{MR2035385} on $R$ and $W$.
From \lemref{der} ($W=R\p$) we know that  our  $W$ is a rank one free $R$-module. So Conditions (1.1), (1.2) in \cite{MR2035385} hold.
Now we prove Condition (1.3).

\begin{lemma}\label{lem8} Suppose that $R=R_m(P)$ with $m\ge 2$ and
$P=t^l(t-a_1)\ldots(t-a_n)$ for some $l\in \Z_+$, $n\in \N$ and
pairwise distinct $a_1,\ldots,a_n\in\C^*$. Then
\begin{itemize}
 \item[(a).] $R=R_m(P)$ is a domain;
 \item[(b).] There exists $T\in \Omega^1$ such that $T(\partial)=1$ and $\Omega^1=R\cdot dR=R\cdot
 T$.
 \item[(c).] $\rho_{\Omega^1}(W)(\Omega^1)=dR=\partial(R)\cdot T$
 \item[(d).] $H^1(\Omega)=\Omega^1/\rho_{\Omega^1}(W)(\Omega^1)=R\cdot dR/dR$ (vector space quotient).
\end{itemize}
\end{lemma}

\begin{proof}(a). We know that $\C[t,t^{-1}]$ is a PID. From  Eisenstein's Criterion, we see that $u^m-P(t)\in \C[ t,t^{-1}][u]$ is irreducible in $\C[u,t,t^{-1}]$. So $R$ is a domain.

(b). Note that $(P, P')=t^k$ for some $k\in \Z_+$. Then there exist
$a(t), b(t)\in \C[t^{\pm1}]$ such that $aP'+bP=1$. Since $W=R\p$,
then $\varphi\in \Omega^1$ is uniquely determined by $\varphi(\p)$.
Let $$T=ma\cdot d(\sqrt[m]{P})+b\sqrt[m]{P}\cdot dt\in R\cdot dR.$$
Then $T(f\partial)=f$ for all $f\in R$ and $\varphi(\p)T=\varphi$.

(c). For all $f, g\in R$, we have
$df(g\partial)=g\partial(f)=\partial(f)T(g\partial)$, which implies
$df=\partial(f)\cdot T$ and $dR=\partial(R)\cdot T$. Then we notice
$$\rho_{\Omega^1}(f\partial)(gT)=d(gT(f\partial))=d(fg),\ \forall\ f,g\in R.$$
Consequently, $\rho_{\Omega^1}(W)(\Omega^1)=dR.$

(d). The first equation follows from the proof of Theorem 7.1 of
\cite{MR2035385} and the second one follows from (b) and (c).
\end{proof}

Before introducing our main result in this section we need to define
$\Deg (f(t)$) for any $f=\sum_{s=s_1}^{s_2} a_s t^s \in
\C[t,t^{-1}]$ with $a_{s_1}a_{s_2}\ne 0$ as follows
$\Deg(f)=(s_1,s_2)$. 
Now we have

\begin{theorem} \label{der.b}  Suppose that $R=R_m(P)$ with $m\ge 2$ and
$P=t^l(t-a_1)\ldots(t-a_n)$ for some $l\in \Z_+$, $n\in \N$ and
pairwise distinct $a_1,\ldots,a_n\in\C^*$. Then the universal central extension of
$\mathcal{R}_m(P)$ is $\mathcal{R}_m(P)\oplus R/\p(R)$ with brackets
 $$[f\p,g\p]=f\p(g)\p-g\p(f)\p+\overline{\p(f)\p(\p(g))}, \forall\,\,f,g\in R,$$
 and $\dim R/\partial(R)=1+n(m-1)$.
\end{theorem}

\begin{proof} From Theorem 7.1 of \cite{MR2035385} we know that the universal central extension of $W$ is
$W\oplus R\cdot dR/dR $ with brackets
 $$[f\p,g\p]=f\p(g)\p-g\p(f)\p+\overline{\psi(f\partial)d\psi(g\partial)}, \forall\,\,f,g\in R,$$
where $\psi:\ W\to R$ is a divergence (see page 87 of
\cite{MR2035385}) and $\overline{\psi(f\partial)d\psi(g\partial)}$
is the canonical image of $\psi(f\partial)d\psi(g\partial)$ in
$H^1(\Omega)=R\cdot dR/dR$.

Define the following linear map
$$\bar\pi: R\cdot dR=R\cdot T \to R/\p(R), \quad \bar\pi(f\cdot T)=\overline{f}.$$
It is clear that $\bar\pi$ is surjective with kernel $\p(R)\cdot T$.
Thus we have the induced space isomorphism
$$\pi: H^1(\Omega)=R\cdot dR/dR=R\cdot T/\p(R)\cdot T \to R/\p(R).$$
Take the divergence $\psi$ as $\psi(f\partial)=\p(f)$, we have
$$\pi(\overline{\psi(f\partial)d\psi(g\partial)})=\pi(\overline{\p(f)d\p(g)})=\pi(\overline{\p(f)\p^2(g)\cdot T})=\overline{\p(f)\p^2(g)}.$$
Identifying $H^1(\Omega)$ with $R/\p(R)$ via $\pi$, we can deduce
the universal central extension.

 Now we only need to verify that $\dim R/\partial(R)=1+n(m-1)$. We
compute

\begin{equation}\label{eq-1}\partial(t^i)= it^{i-1} P^{\frac{m-1}{m}},\forall i\in \Z,\end{equation}\begin{equation}\label{eq-2} \partial(t^{i}P^{\frac{k}{m}})=(iP+\frac{k}{m}t P')t^{i-1}P^{\frac{k-1}{m}},\forall i\in \Z, k=1,2,\ldots,m-1.\end{equation}

Thus we only need to verify that for any $k=1,2,\ldots,m-1$,
$$V_{k}={\rm span}_{\C}\{(iP+\frac{k}{m} t P')t^{i-1}|i\in \Z\}$$
has codimension $n$ in $\C[t,t^{-1}]$. Let $i_k=\lfloor -\frac{kl}
{m} \rfloor$, and $$f_{i,k}=(iP+\frac{k}{m} t
P')t^{i-1}+\delta_{i,-\frac{k(n+l)}
{m}}t^{i-1+n+l}+\delta_{i,-\frac{kl}{m} } t^{i-1+l}, \forall i\in
\Z.$$ We see that  $\{f_{i,k}|i\in \Z\}$ is linearly independent for
$k=1,2,\ldots,m-1$, and
 Deg$(f_{i,k})=(i-1+l,i-1+l+n)$.

 For any nonzero linear combination $f$ of $\{f_{i,k}|i\in \Z\}$ let Deg$(f)=(s_1,s_2)$. Then $s_2-s_1\ge n$. Thus for any $k=1,2,\ldots,m-1$, we know that $$\{f_{i,k}|i\in \Z\}\cup \{t^{i_k+l-1},t^{i_k+l},\ldots,t^{i_k+l+n-2}\}$$ is a basis of $\C[t,t^{-1}]$.  Note that if $\delta_{i,-\frac{k(n+l)} {m}}t^{i-1+n+l}+\delta_{i,-\frac{kl}{m} } t^{i-1+l}\ne 0$, the missing term of
 $f_{i,k}$  belongs to span$\{t^{i_k+l-1},t^{i_k+l},\ldots,t^{i_k+l+n-2}\}$. So for any $k=1,2,\ldots,m-1$, $$\{iP+\frac{k}{m} t P')t^{i-1}|i\in \Z\}\cup \{t^{i_k+l-1},t^{i_k+l},\ldots,t^{i_k+l+n-2}\}$$ is also a basis of $\C[t,t^{-1}]$. Combining with (\ref{eq-1}) and (\ref{eq-2}), we deduce that $\dim R/\partial(R)=1+n(m-1)$. This completes the proof.
\end{proof}

From \thmref{der.b} we need   to compute $\overline{f(t)}\in
R/\partial R$ for any $f(t)\in R$. This is actually not easy.

For  the Lie algebras $\mathcal{S}_m(P)$ we have similar results as in \lemref{lem8}. Using these results
we can prove a similar result for the Lie algebras $\mathcal{S}_m(P)$  as in \thmref{der.b}.
\begin{theorem}\label{der-1}  Suppose that  $m\ge 2$ and
$P\in \C[t]$ is a polynomial of degree $n\ge 1$ without multiple roots. Then $\dim S/\partial(S)=(n-1)(m-1)$ and the universal central extension of
$\mathcal{S}_m(P)$ is $\mathcal{S}_m(P)\oplus S/\p(S)$ with brackets
 $$
 [f\p,g\p]=f\p(g)\p-g\p(f)\p+\overline{\p(f)\p(\p(g))}, \forall\,\,f,g\in S.
 $$\end{theorem}

 \begin{proof} The proof is similar to that of Theorem \ref{der.b}. Here we only compute $\dim S/\partial(S)$. Similarly we have
 \begin{equation}\label{eq-3}\partial(t^i)= it^{i-1} P^{\frac{m-1}{m}},\forall i\in \Z_+,\end{equation}\begin{equation}\label{eq-4} \partial(t^{i}P^{\frac{k}{m}})=(iP+\frac{k}{m}t P')t^{i-1}P^{\frac{k-1}{m}},\forall i\in \Z_+, k=1,2,\ldots,m-1.\end{equation}

 Note that   ${\rm deg}((iP+\frac{k}{m}t P')t^{i-1})=n+i-1$ for all $i\in \Z_+$ and $k=1,2,\ldots,m-1$. Thus $\C[t]$ has  basis $\{(iP+\frac{k}{m}t P')t^{i-1}|i\in \Z_+\}\cup \{1,t, \ldots,t^{n-2}\}$ for all $k=1,2,\ldots,m-1$. Combining with (\ref{eq-3}) and (\ref{eq-4}), we have $\dim S/\partial(S)=(m-1)(n-1)$.
 \end{proof}

\section{Invertible elements of $R_2(P)$  and ${S}_2({P})$}

To study the isomorphisms and automorphisms among the various Lie
algebras $\mathcal{R}_m({P})$ (resp. $\mathcal{S}_m({P})$), from
\cite{MR966871} we need only to consider isomorphisms and
automorphisms of the Riemann surfaces  ${R}_m({P})$ (resp.
${S}_m({P})$). This is a very hard problem (See \cite{MR2203507}).
So far we actually know the  automorphism groups of compact Riemann
surfaces with genus $0$ or $1$ only, and there are scattered results
on higher genus Riemann surfaces, see \cite{MR2203507, MR1796706}
and the references therein. We point out that the Riemann surfaces
${R}_m({P})$ (resp. ${S}_m({P})$) are generally not compact which makes the problem even harder.

So we will make some attempts in the rest of the paper to study
isomorphisms and automorphisms among some of the  hyperelliptic Lie
algebras $\mathcal{R}_2({P})$ (resp. $\mathcal{S}_2({P})$) whose
corresponding Riemann surfaces ${R}_2({P})$ and ${S}_2({P})$ have
higher genus in general.

We will first determine the unit group $R_2^*(P)$ (resp. $S^*_2(P)$)
for some $R_2(P)$ (resp. $S_2(P)$), i.e., the multiplicative group
consisting of all invertible elements in $R_2(P) $ (resp. $S_2(P)$)
for some special polynomial $P\in\C[t]$. {We always assume
that both Lie algebras $\mathcal{R}_2({P})$ and $\mathcal{S}_2({P})$
are simple in this section, and use the realizations
$ {R}_2({P})=\C[t^{\pm1}, \sqrt{P}]$ and $ {S}_2({P})=\C[t, \sqrt{P}]$. The conjugate of
$X=f+g\sqrt{P}\in R_2(P)$ is defined as $\overline{X}=f-g\sqrt{P}$.} 

\begin{lemma} \label{relprimelemma}
(a). The unit group  $R_2^*(P)$ of $R_2(P)$ is 
$$\{t^i\,|\,i\in\Z\}\cdot\{f+g\sqrt{P}\ |\ f,g\in\C[t], f^2-g^2P=c t^k\ \text{for\ some\ } c\in\C^*\ \text{and}\
k\in\Z_+\}.$$

 (b). The unit group of $S_2(P)$ is
%
$$S^*_2(P)=\{f+g\sqrt{P}\ |\ f,g\in\C[t], f^2-g^2P=c {\text{ for some }} c\in\C^* \}.$$
\end{lemma}

\begin{proof} (a). For convenience, we denote
$$M=\{f+g\sqrt{P}\ |\ f,g\in\C[t], f^2-g^2P=c t^k\ \text{for\ some\ } c\in\C^*\ \text{and}\
k\in\Z_+\}$$ It is easy to see that $M\subseteq R^*_2(P)$.

Now suppose $f+g \sqrt{P} \in R_2^*(P)$ where $f,g\in \C[t^{\pm1}]$.
If $g=0$, then it is obvious that $f$ is a nonzero multiple of some
$t^k, k\in\Z$. Now suppose $g\neq 0$. Multiplying $f+g\sqrt{P}$ by a
suitable $t^l, l\in\Z$, we may assume that $f, g\in \C[t]$ and
$\gcd(f,g)=1$. There exists $t^{-k}(x+y\sqrt{P})$ where $k\in\Z_+$,
$x,y\in\C[t]$ with either $x(0)\neq 0$ or $y(0)\neq 0$ such that
$$
1=t^{-k}(x+y\sqrt{P})(f+g\sqrt{P})=t^{-k}\big(fx+gyP+(fy+gx)\sqrt{P}\big)
$$
We obtain that
$$fx+gyP=t^k, \quad fy+gx=0,$$
from which we see that $\gcd(x,y)=1$. Then there exists some
$c\in\C^*$ such that $x=cf$ and $y=-cg$. Now using $fx+gyP=t^k,$ we
get
\begin{equation}\label{Laurentuniteqn}
t^k=cf^2-cg^2P=c(f^2-g^2P),
\end{equation}
forcing $f+g\sqrt{P}\in M$. Therefore (a) follows.

(b). The proof for this part is similar to that of (a).
\end{proof}

From the above lemma, we see that to determine the unit group
$R^*_2(P)$ is equivalent to solving the Laurent polynomial equations
$f^2-g^2P=t^k$ for $k\in\Z_+$, and to determine the unit group
$S^*_2(P)$ is equivalent to solving the polynomial Pell equation
$f^2-g^2P=1$. The study on solutions of the polynomial Pell equation
$$ f^2-g^2P=1$$
for a given $P\in\C[t]$ is a very famous and very hard problem (See
\cite{MR2183270}). We remark that when $P=t^2-1$, the solutions to
the above Pell equation are the famous Chebyshev polynomials (See
\cite{MR1937591}).

\begin{lemma}\label{oddposdegreet}
Suppose that $\deg(P)$ is odd. Then $S^*_2(P)=\C^*$. If furthermore
$P(t)$ has no multiple roots and $t | P$, then
\begin{itemize}
\item[(i).] The polynomial equation $ f^2-g^2P=t^{k} $ for $k\geq 0$ has solutions for
$f,g\in\C[t]$ with $\gcd(f,g)=1$ and $g\ne0$ if and only if
$\deg(P)=1$ and $k=1$; Moreover, the only solutions for the equation
$f^2-g^2t=t$ are $f=0, g=\pm\sqrt{-1}$.
\item[(ii).] $R^*_2(P)=\bigcup_{k\in\Z}\C^*t^k\cong \mathbb C^*\times \mathbb Z$ if $\deg(P)\geq 3$,
and $$R^*_2(P)=\bigcup_{k\in\Z}\Big(\C^*t^k\cup
\C^*t^{k+\frac{1}{2}}\Big)\cong\C^*\times\Z\times(\Z/2\Z)$$
if $P=c^2 t$ for some $c\in\C^*$.
\end{itemize}
\end{lemma}

\begin{proof} The result $S^*_2(P)=C^*$ is well-known (See \cite{MR2183270}).
Part (ii) follows directly from (i). For part (i), we suppose that
$f, g$ are solutions to the equation $f^2-g^2P=t^k$ for some $ k\in\Z_+$ with
$\gcd(f,g)=1$.

If $k\geq 2$, then $t|P$ implies $t|f$. We assume that $P=tQ$ and
$f=th$ for some $Q, h\in\C[t]$. Hence
$$
th^2=g^2Q+t^{k-1}.
$$
This means that $t$ divides $g$. Then $g$ and $f$ are not
relatively prime, a contradiction.

Suppose now $k\leq 1$. If $\deg(P)\geq 3$, then
$\deg(g^2P+t^k)=\deg(g^2P)$, which by assumption is odd. However
$\deg(f^2)$ is even, impossible.

If $\deg(P)=1$ and $k=0$, we can deduce contradiction similarly as
above. Now suppose $\deg(P)=1$, say, $P=c^2 t$ for some $c\in\C^*$,
and $k=1$. It is easy to see that the only solutions for
$f^2-c^2g^2t=t$ are given by $f=0$ and $g=\pm\sqrt{-1} c^{-1}$.
\end{proof}

Next we turn our attention to the most interesting case studied by
Date, Jimbo, Kashiwara and Miwa \cite{DJKM} where they investigated
integrable systems arising from Landau-Lifshitz differential
equation. Let $$\displaystyle{P(t)=\frac{t^4-2\beta
t^2+1}{\beta^2-1}},\ \ \beta\neq \pm 1.$$ Observe that in this case
$P(t)=q(t)^2-1$ where $q(t)=\frac{t^2-\beta}{\sqrt{\beta^2-1}}$, and
we    will see that the group $R_2^*(P)$ is quite big.

For convenience, let
\begin{equation*}\begin{split}
& \lambda_0=\frac{t^2-\beta}{\sqrt{\beta^2-1}}+\sqrt{P},\\
& \lambda_1=\frac{t^2+1}{\sqrt{2(\beta+1)}}+\sqrt{\frac{\beta-1}{2}}\sqrt{P},\\
& \lambda_2=\frac{t^2- 1}{\sqrt{2(\beta-1)}}+\sqrt{\frac{\beta+1}{2}}\sqrt{P}.\\
& \lambda_3=\frac{\beta t^2-1}{\sqrt{\beta^2-1}}+\sqrt{P}.\\
\end{split}\end{equation*}
It is easy to verify that $\lambda_0, \lambda_1, \lambda_2,
\lambda_3\in R_2^*(P)$ and $\lambda_0 \in S_2^*(P)$. Actually,
$\lambda_0 \bar\lambda_0=1$, $\lambda_1\bar\lambda_1=t^2$,
$\lambda_2\bar\lambda_2=t^2$, $\lambda_1\lambda_2=t^2\lambda_0$, and
$\lambda_1\bar\lambda_2=\lambda_3$.

For later use now we define a non-linear operator $\tau$ on $S_2(P)$
as follows. For $f(t)+g(t)\sqrt{P}\in S_2(P)$ with $f, g\in\C[t]$,
we define
$$\tau(f(t)+g(t)\sqrt{P})=t^r(f(t^{-1})+g(t^{-1})t^{-2}\sqrt{P})\in
S_2(P),$$ where $r=\max\{\deg(f), \deg(g)+2\}$. 
We can easily obtain that
$$\tau(t^m(f+g\sqrt{P}))=\tau(f+g\sqrt{P}),\ \forall\ m\in\Z_+$$
and
$$\tau^2(f+g\sqrt{P})=f+g\sqrt{P},\ \ \text{if}\ \ t\nmid\gcd(f,g).$$
Moreover, we have
$$\tau(\lambda_0)=-\bar\lambda_1\lambda_2,\ \
\tau(\lambda_1)=\lambda_1,\ \ \tau(\lambda_2)=-\bar\lambda_2.$$ For
any $X, Y\in S_2(P)$ we have $\tau(XY)=t^k\tau(X)\tau(Y)$ for some
$k\in\Z$.
\begin{theorem}
Let $P(t)=\frac{t^4-2\beta t^2+1}{\beta^2-1}$ where $\beta\ne\pm1$.
\begin{itemize}\item[(a).]  As a multiplicative group, $R_2^*(P)$ is generated by $\C^*$, $t,
 \lambda_1, \lambda_2$.
\item[(b).] $R_2^*(P)\simeq\mathbb C^*\times\Z\times\Z\times\Z.$
\item[(c).]
$S_2^*(P)= \mathbb C^*\cdot \{ \lambda_0^i |\,i \in\mathbb
Z\}\simeq\mathbb C^*\times\Z.$
\end{itemize}
\end{theorem}

\begin{proof} (a). By \lemref{relprimelemma}, 
we need to find all polynomial solutions $(f,g)$ to the equation $f^2-g^2P=t^k$
for ay fixed $k\in\mathbb \Z_+$ with $\gcd(f,g)=1$ and $g\neq 0$. We will use
the following facts frequently:
\begin{equation}\label{eq-09}(f+g\sqrt{P})(q\pm \sqrt{P})=(fq\pm gP)+(gq\pm f)\sqrt{P},\end{equation}
\begin{equation}\label{eq-10}\gcd(fq\pm gP,gq\pm f)=1,\end{equation}
\begin{equation}\label{eq-11}(fq\pm gP)^2-(gq\pm f)^2P=t^k,\end{equation}
\begin{equation}\label{eq-deg}
(gq-f)(gq+f)=g^2q^2-f^2=g^2(P+1)-f^2=g^2-t^k.
\end{equation}
From (\ref{eq-09})--(\ref{eq-11}), replacing $(f+g\sqrt{P})$ with
$(f+g\sqrt{P})\lambda_0^m$ for suitable $m\in \Z$ if necessary, we
may assume that
\begin{equation}\label{eq-12}\deg(gq\pm f)\ge \deg g.\end{equation}

And if $\deg g^2\ge k$, then from (\ref{eq-deg}) and (\ref{eq-12}),
we get $\deg(gq-f)=\deg(gq+f)=\deg g$ or $\deg g=k=0$. Noting that $\deg(q)=2$ we cannot cancel the highest terms of $gq$ and of $f$ in both   $gq-f$ and $gq+f$. The first case cannot occur.   So we have proved
\begin{equation}\label{eq-13}2\deg(g)<k,\,\, {\rm{or}}\,\, \deg g=k=0.\end{equation}
Similarly from (\ref{eq-deg}) and (\ref{eq-12}) we may deduce that
\begin{equation}\label{eq-14}k\ne 2\deg(g)+1.\end{equation}
In particular, there is no such resolution if $k=1$.

%

\vskip 5pt
\noindent{\bf Case 1}: $k=2l+1$ for some $l\in \N$.

From (\ref{eq-13}) and (\ref{eq-14}), we see that  $\deg g<l$.
Then from (\ref{eq-12}) and (\ref{eq-deg}), we have $\deg f=\deg g+2=l+1$. As $P(1/t)=t^{-4}P(t)$ we get
$$
t^{2l+2}f(1/t)^2-t^{2l-2}g(1/t)^2P(t)=t.
$$
Now $t^{l-1}g(1/t)$ and $t^{l+1}f(1/t)$ are relatively prime in
$\mathbb C[t]$. Hence we reduce to the case that $k=1$ for which
there is no such solution.

\noindent{\bf Case 2:} $k=0$.

From (\ref{eq-13}), we have $g\in \C^*$, and it is easy to see that
the solutions for $f^2=1+g^2P$ are exactly $g=\pm1$ and $f=\pm q$.
Hence $f+g\sqrt{P}=\pm\lambda_0^{\pm1}$.

\noindent{\bf Case 3:} $k=2$.

From (\ref{eq-13}), we have $g=c\in \C^*$. Now we solve for $f$ as
follows:
\begin{align*}
f^2&=t^2+c^2\frac{(t^4-2\beta t^2+1)}{\beta^2-1} \\
&=\frac{c^2t^4+(\beta^2-1-2\beta c^2)t^2+c^2 }{\beta^2-1}
\end{align*}
which is a square in $\C[t]$ if and only if
$$
2\beta c^2\pm 2c^2-\beta^2+1=0.
$$
Solving this equation for $c$ we get
$$
g^2=c^2=\frac{\beta\pm 1}{2},\qquad \text{and}\qquad f^2=
\frac{(t^2\mp 1)^2}{2(\beta\mp 1)}.
$$
Therefore
 \begin{equation}f+g\sqrt{P}\in \{\pm \lambda_0^i\lambda_1,\pm \lambda_0^i\lambda_2,
 \pm t^2\lambda_0^i\lambda_1^{-1},\pm t^2 \lambda_0^i\lambda_2^{-1}\,|\, i\in\Z\}.\end{equation}





\noindent{\bf Case 4:} $k=2l$, for some $ l\ge2$. \vskip 5pt

Again from (\ref{eq-13}), we have $\deg g<l$.
If $\deg g=l-1$, then $\deg f=l+1$ and
$$
(t^{l+1}f(t^{-1}))^2-(t^{l-1}g(t^{-1}))^2P(t)=t^2,
$$
 where we have used $t^4P(t^{-1})=P(t)$. Now $f(t^{-1})t^{l+1}$
 and $g(t^{-1})t^{l-1}$ are relatively prime in $\C[t]$. Hence we reduced this case to $k=2$. So we have that
\begin{equation*}\begin{split}
\tau(f+g\sqrt P)=f&(t^{-1})t^{l+1}+g(t^{-1})t^{l-1}\sqrt{P} \\
& \in\{\lambda_0^m|m\in \Z\}\cdot \{\pm \lambda_1,\pm \lambda_2, \pm
t^2\lambda_1^{-1},\pm t^2
\lambda_2^{-1}\}.\end{split}\end{equation*} Thus
\begin{equation}f+g\sqrt{P}=\tau^2(f+g\sqrt{P})\in \{t^{i}\lambda_0^{j}\lambda_3^{k}|i,j,k\in \Z\}\cdot \{\pm \lambda_1^{\pm 1},\pm \lambda_2^{\pm 1}\}. \end{equation}

Now we only need to deal with the case $\deg g\le l-2$. If $\deg
g=l-2$, then $\deg f\leq l$; if $\deg g<l-2$, then $\deg f=l$. In
both cases, we have
$$
(t^lf(t^{-1}))^2-(t^{l-2}g(t^{-1}))^2P(t)=1,
$$
where $f(t^{-1})t^{l}$ and $g(t^{-1})t^{l-2}$ are relatively prime
in $\C[t]$. Hence we reduced this case to Case 2: $k=0$. So
$$\tau(f+g\sqrt P)=f(t^{-1})t^l+g(t^{-1})t^{l-2}\sqrt{P}=\pm
(q+\sqrt{P})^r=\pm\lambda_0^r$$ for some $r\in \Z$, i.e.
\begin{equation}f+g\sqrt{P}\in \pm \{t^i\lambda_0^j\lambda_3^k\,|\,i,j,k\in \Z\}.\end{equation}
%
%
So we have proved (a).

\

(b). For any $X\subset R_2^*(P)$, let $\langle  X\rangle$ be the
subgroup of $R_2^*(P)$ generated by $X$. It is clear that $\C^*\cap
\langle t \rangle=\{1\}$, hence $\langle\C^*, t \rangle=\C^*\times
\langle t \rangle$. Since $\lambda_1(0)\ne0$ and $\lambda_2(0)=0$,
then $ \langle \lambda_1, \lambda_2\rangle=\langle
\lambda_1\rangle\times \langle \lambda_2\rangle$.

Now we need only  to show that
\begin{equation} \langle \C^*, t  \rangle \cap
\langle \lambda_1,\lambda_2 \rangle=\{1\}.\end{equation} Suppose we
have $ \lambda_1^i\lambda_2^j=ct^k$ for some $i,j,k\in\Z$ and
$c\in\C^*$. We may assume that $ik\ne0$ and $i>0$. (The argument for $jk\ne0$ is similar).

{\bf Case 1.} $j\ge0$.

We see that $k>0$.   From $ \lambda_1^i\lambda_2^j=ct^k$ we know
that $\bar \lambda_1^i\bar\lambda_2^j=ct^k$.  Multiplying them
together we obtain that $k=i+j$ and $c=\pm 1$. From $
\lambda_1^i\lambda_2^j=\pm t^{i+j}$ we $ \lambda_1^it^{2j}=\pm
t^{i+j}\bar\lambda_2^j$. Since $\lambda_1(0)\bar\lambda_2(0)\ne0$ we
deduce that $i=j$. Going back to $ \lambda_1^i\lambda_2^i=\pm
t^{2i}$ we deduce that $\lambda_0^i=\pm 1$, which is impossible
since $\lambda_0(0)\ne0$. So this case does not occur.

{\bf Case 2.} $j<0$.

If $k>0$, we have  $ \lambda_1^i=c\lambda_2^{-j}t^k$. Taking the
value at $t=0$ for both sides we see a contradiction. So this case
does not occur.

Now $k<0$. We have $ \lambda_1^it^{-k}=c\lambda_2^{-j}$ whose
conjugates are $ \bar\lambda_1^it^{-k}=c\bar\lambda_2^{-j}$. Taking
the value at $t=0$ for both sides we see a contradiction. So this
case does not occur either.

Thus (4.12) holds and (b) follows.

 Part (c) is well known (See
\cite{MR2183270}).
\end{proof}

\section{Isomorphisms and automorphisms of   $R_2(P)$}

In this section we will determine  the necessary and sufficient
conditions for $\mathcal{R}_2({P_1})\simeq
\mathcal{R}_2({P_2})$, in the case  $P_1(t)$ and $P_2(t)$ are separable polynomials of odd degree with one root at $t=0$, then  describe the possible automorphism groups of $\mathcal{R}_2(P)$.

\begin{theorem}\label{isothm} Suppose the polynomials $P_1(t_1)=t_1(t_1-a_1)\cdots (t_1-a_{2n})$
and $P_2(t_2)=t_2(t_2-b_1)\cdots (t_2-b_{2m})$ both have distinct
roots where $a_i, b_j\in\C^*$.  Let  $\phi: R_2({P_1})\to
R_2({P_2})$ be an isomorphism. Then $m=n$ and there is some $c\in\C^*$ and $\gamma\in S_{2n}$ such that
$$
a_{\gamma(i)}=c^2b_{i},\ \forall\ i=1,2,\cdots, 2n, {\rm{\text{ or }}}
$$
$$
a_ib_{\gamma(i)}=c^2,\ \forall\ i=1,2,\cdots, 2n,
$$
where $S_{2n}$ is the symmetry group on $\{1,2,\cdots, 2n\}$; and
$$\phi(t)=c^2t  ,\,\,\,\phi(\sqrt{P_1})=\pm c ^{2n+1}\sqrt{P_2}, \,\,{\text or } $$
$$\phi(t)=c^2t^{-1}  ,\,\,\,\,\,\phi(\sqrt{P_1})=\pm ct^{-n-1}\sqrt{a_1a_2...a_{2n}P_2}.$$
\end{theorem}

\begin{proof}   We know that $\phi(\C)=\C$ and that $\phi$ is uniquely
determined by $\phi(t_1)$ and $\phi(\sqrt{P_1(t_1)})$. We deduce from
\lemref{oddposdegreet}  that $\phi(t_1)=c ^2t_2^{\pm 1}$ as $\phi$ must
map the generator $t_1$ of $R_2({P_1})^*/\C^*$ to a generator of
$R_2({P_2})^*/\C^*$.  Thus $\phi$ maps $\mathbb C[t_1,t_1^{-1}]$
bijectively onto $\mathbb C[t_2,t_2^{-1}]$ and hence it maps
$\mathbb C[t_1,t_1^{-1}]\sqrt{P_1(t_1)}$ bijectively onto $\mathbb
C[t_2,t_2^{-1}]\sqrt{P_2(t_2)}$ since $$\mathbb
C[t_{i}^{\pm1}]\sqrt{P_i(t_i)}=\{x\in R_2(P_i)\setminus
\C[t_i^{\pm1}]\,|\, x^2\in \C[t_i^{\pm1}]\},\ i=1,2.$$ So
$\phi(\sqrt{P_1(t_1)})=f\sqrt{P_2(t_2)}$ for some $f\in
\C^*\{t^k\,|\, k\in\Z\}$.

\vskip 5pt \noindent {\bf Case 1}: $\phi(t_1)=c^2t_2$ and
$\phi(\sqrt{P_1(t_1)})=f\sqrt{P_2(t_2)}$ for some  $c\in\C^*$ and
some $f\in \C^*\{t^k\,|\, k\in\Z\}$.

We have
\begin{align*}
\phi(P_1(t_1))&=\phi(t_1(t_1-a_1)\cdots (t_1-a_{2n})) \\
             &=c^2t_2(c^2t_2-a_1)\cdots (c^2t_2-a_{2n}) \\
             &=c^{4n+2}t_2(t_2-(a_1/c^2))\cdots (t_2-(a_{2n}/c^2)) \\
             &=f^2P_2(t_2) \\
             &=f^2t_2(t_2-b_1)\cdots (t_2 -b_{2m}).
\end{align*}
Now since the $a_i/c^2\neq 0$, $i=1,\dots, 2n$ are distinct, one must
have that $f$ is constant such that $f^2=c^{4n+2}$, $m=n$ and
$b_i=a_{\gamma(i)}/c^2, i=1,\cdots, 2n$ for some $\gamma\in S_{2n}$.

Conversely, given any $c\in \C^*$ and $\gamma\in S_{2n}$ with
$b_i=a_{\gamma(i)}/c^2$ for all $i=1\cdots,2n$, we define an
endomorphism $\phi:\ R_2(P_1) \to R_2(P_2)$ by $\phi(t_1)=c^2t_2$ and
$\phi(\sqrt{P_1(t_1)})=\pm\sqrt{c^{4n+2}P_2(t_2)}$. Then
$\phi(P_1(t_1))=c^{4n+2}P_2(t_2)$ and hence $\phi$ is an
isomorphism.

\vskip 5pt \noindent {\bf Case 2}: $\phi(t_1)=c^2t_2^{-1}$ and
$\phi(\sqrt{P_1(t_1)})=f\sqrt{P_2(t_2)}$ for some  $c\in\C^*$ and
some $f\in \C^*\{t^k\,|\, k\in\Z\}$.

One has
\begin{align*}
\phi(P_1(t_1))&=\phi(t_1(t_1-a_1)\cdots (t_1-a_{2n})) \\
             &=c^2t_2^{-1}(c^2t_2^{-1}-a_1)\cdots (^2ct_2^{-1}-a_{2n}) \\
             &=c^2a_1\cdots a_{2n}t_2^{-2n-1}(t_2-(c^2/a_1))\cdots (t_2-(c^2/a_{2n})) \\
             &=f^2t_2(t_2-b_1)\cdots (t_2 -b_{2m}).
\end{align*}
Now since all $c^2/a_i\neq0, i=1,\dots, 2n$ are distinct, one must
have that $m=n$, $f^2=c^2a_1\dots a_{2n}t^{-2n-2}$ and
$c^2=a_ib_{\gamma(i)}$ for some $\gamma\in S_{2n}$. 
%
%

Similar to the previous case, given any $c\in \C^*$ and $\gamma\in
S_{2n}$ with $a_ib_{\gamma(i)}=c^2$, we can define an isomorphism
$\phi:\ R_2(P_1) \to R_2(P_2)$ by $\phi(t_1)=c^2 t_2^{-1}$ and
$\phi(\sqrt{P_1(t_1)})=\pm c t^{-n-1}\sqrt{a_1\dots a_{2n}P_2(t_2)}$.
\end{proof}

\begin{corollary} \label{firstred} Suppose the polynomial $P(t)=t(t-a_1)\cdots (t-a_{2n})$ has distinct roots.
There are two possible types of automorphisms  $\phi\in\textrm{Aut}(R_2(P))$:
\begin{enumerate}
\item If $a_{\gamma(i)}=\a^2 a_i$ for some $4n$-th root of unity $\a$ and $\gamma\in
S_{2n}$, then
$$
\phi(t)=\a^2 t,\ \ \phi(\sqrt{P(t)})=\pm \a \sqrt{ P(t)}.
$$
Denote these automorphisms by $\phi_{\a}^{\pm}$ respectively. It
is clear that $\big(\phi_{\a}^{\pm}\big)^{4n}=\id$.
\item  If $a_ia_{\gamma(i)}=c^2$ for some $c\in \C^*$ and $\gamma\in
S_{2n}$ 
then $\phi(t)=c^2 t^{-1}$, and
$$\phi(\sqrt{P(t)})=\pm t^{-n-1}c^{n+1}\sqrt{P(t)}\quad \text{if}\ \prod_{i=1}^{2n}a_i=c^{2n},$$
or
$$\phi(\sqrt{P(t)})=\pm t^{-n-1}c^{n+1}\sqrt{-P(t)}\quad \text{if}\ \prod_{i=1}^{2n}a_i=- c^{2n}.$$
Denote these automorphisms by $\psi_{c}^{\pm}$ respectively in both
cases. It is clear that $\big(\psi_c^+\big)^2=\id$ and $\big(\psi_c^-\big)^2=\varphi_1^-$.
\end{enumerate}

\end{corollary}

\begin{proof} We continue the proof of the previous theorem for
automorphisms. In the first case, we use $\a^2$ instead of $c$.
Hence we have $\phi(t)=\a^2 t$ for some $\a\in\C^*$ and
$$
\a^{4n}=\frac{\prod_{i=1}^{2n}a_i}{\prod_{i=1}^{2n}a_{\gamma(i)}}=1,\,\,f=\pm
\a^{2n+1}. 
$$
so that $\a$ is a $4n$-th root of unity (not necessarily primitive). So
$\phi^{4n}=\id$.

In the second case, we use $c^2$ instead of $c$. 
We see that  $\prod_{i=1}^{2n}a_ia_{\gamma(i)}=c^{4n}$ and $f^2=c^2a_1\dots
a_{2n}t^{-2n-2}$ so that
$$f=\pm t^{-n-1} { c^{n+1}}\ \text{and}\ \phi^2=\id, \quad \text{if}\ \prod_{i=1}^{2n}a_i=c^{2n}$$
or
$$f=\pm \sqrt{-1}t^{-n-1}{c^{n+1}}\ \text{and}\ \phi^2=\phi_{1}^{-}, \quad \text{if}\ \prod_{i=1}^{2n}a_i=- c^{2n}.$$
Our result follows directly. 
\end{proof}

%

\begin{corollary}\label{automorphismthm} Suppose $P(t)=t(t-a_1)\cdots (t-a_{2n})$ has distinct roots.
\begin{enumerate}
\item If there does not exist $\psi_{c}^{\pm}\in \Aut(\mathcal{R}_2(P))$ for any $c\in\C^*$,
 then  $\Aut(\mathcal{R}_2({P}))$ is generated by $\varphi^+_\a$   for a primitive root of unity of order $2k$ with $k|2n$.
 Consequently, $\Aut(\mathcal{R}_2({P}))\simeq \Z_{2k}$.
 \item If there exists $\psi_{c}^{\pm}\in \Aut(\mathcal{R}_2(P))$ for some $c\in\C^*$,
 then $\Aut(\mathcal{R}_2({P}))$ is generated by  $\varphi_{\a}^{+}$, $\psi_c^{\pm}$  where  $\a$ is a primitive root of unity  of order  $2k$ with  $k|2n$. (The group structure of $\Aut(\mathcal{R}_2({P}))$ will be explained later.)

\end{enumerate}
\end{corollary}

\begin{proof} We first consider (1).
From \Corref{firstred}, we know that there is an automorphism
$\varphi^+_\a$ of maximal order  $2k$ with   $k|2n$. Then
it is easy to see that any other automorphisms coming from Corollary
\ref{firstred} (1) are of the form $\varphi^+_{\a^j},
j=0,1,\cdots,2k-1$. Noticing that
$\varphi^-_\a=\varphi^+_{\a^{k+1}}$, we have
$$ \{\varphi^+_{\a^j}\ |\ i=0,\cdots,2k-1\}= \Aut(\mathcal{R}_2(P)).$$
Thus (1) follows.

%

For (2), we note that if $\sigma_1, \sigma_2$ are automorphisms
coming from \Corref{firstred} (2), then $\sigma_1\sigma_2^{-1}$ is
an automorphism corresponding to \Corref{firstred} (1). Thus
$\Aut(\mathcal{R}_2({P}))$ is generated by     $\varphi_{\a}^{+}$, $\psi_c^{\pm}$  where  $\a$ is a primitive root of unity  of order $2k|2n$.
%
%
%
\end{proof}

\noindent {\bf Acknowledgments.} The research presented in this
paper was carried out during the visit of R.L. and B.C. to Wilfrid
Laurier University supported by University Research Professor  grant in August of 2014. X.G. is partially supported by
NSF of China (Grant 11101380, 11471294) R.L. is partially supported
by NSF of China (Grant 11471233, 11371134) and Jiangsu Government Scholarship for
Overseas Studies (JS-2013-313). K.Z. is partially supported by  NSF
of China (Grant 11271109) and NSERC.

%

\def\cprime{$'$} \def\cprime{$'$} \def\cprime{$'$}
\providecommand{\bysame}{\leavevmode\hbox
to3em{\hrulefill}\thinspace}
\providecommand{\MR}{\relax\ifhmode\unskip\space\fi MR }
\providecommand{\MRhref}[2]{%
  \href{http://www.ams.org/mathscinet-getitem?mr=#1}{#2}
} \providecommand{\href}[2]{#2}

%
%
%
%
%

\end{document}